\documentclass[12pt,twoside]{amsart}
\usepackage{amsmath}
\usepackage{amssymb}
\usepackage{amsthm}
\usepackage{latexsym}
\usepackage{amscd}
\usepackage{xypic}
\xyoption{curve}
\usepackage{ifthen} 
\usepackage{hyperref} 
\usepackage{graphicx}

\newtheorem{theorem}{Theorem}[section]
\newtheorem{lemma}[theorem]{Lemma}
\newtheorem{proposition}[theorem]{Proposition}

\newtheorem{remark}[theorem]{Remark}
\theoremstyle{definition}

\newcommand {\bZ}{\mathbb{Z}}

\newcommand {\bC}{\mathbb{C}}
\newcommand {\bP}{\mathbb{P}}

\newcommand{\cM}{{\mathcal M}}
\newcommand{\cO}{{\mathcal O}}

\newcommand{\GL}{\operatorname{GL}}
\newcommand{\SL}{\operatorname{SL}}
\newcommand{\PGL}{\operatorname{PGL}}
\newcommand{\im}{\operatorname{im}}

\def\p#1{{\bP^{#1}_{\bC}}}

\def\Ofa#1{{\cO_{#1}}}
\def\ga#1{{{\accent"12 #1}}}

\newcommand{\PP}{\mathbb{P}} 
 
\newcommand{\ZZ}{\mathbb{Z}} 
\newcommand{\CC}{\mathbb{C}}

\newcommand{\xycenter}[1]{\begin{center}\mbox{\xymatrix{#1}}\end{center}} 
\newboolean{xlabels} 
\newcommand{\xlabel}[1]{ 
                        \label{#1} 
                        \ifthenelse{\boolean{xlabels}} 
                                   {\marginpar[\hfill{\tiny #1}]{{\tiny #1}}} 
                                   {} 
                       } 
\setboolean{xlabels}{false} 

\title[Moduli spaces related to tetragonal curves]{BIRATIONAL PROPERTIES OF SOME MODULI SPACES
RELATED TO TETRAGONAL CURVES OF GENUS $7$}
 
\subjclass[2000]{Primary 14H10; Secondary 14E08 14H45 14H50 14L24 14M20}

\keywords{}

\author[Ch. B\"ohning, H.-Chr. Graf von Bothmer, G.Casnati]{Ch. B\"ohning$\ ^{\dag}$, H.-Chr. Graf von Bothmer$\ ^{\dag}$ \and G. Casnati$\ ^{\ddag}$}
\thanks{$\ ^{\dag}$  Supported by the German Research Foundation (Deutsche Forschungsgemeinschaft (DFG)) through the Institutional Strategy of the University of G\"ottingen}
\thanks{$\ ^{\ddag}$ Supported by the framework of PRIN 2008 \lq Geometria delle variet\ga a algebriche e dei loro spazi di moduli\rq, cofinanced by MIUR
}

\date{\today}

\begin{document}

\maketitle

\begin{abstract}
Let $\cM_{7,n}$ be the (coarse) moduli space of smooth curves of genus $7$ with $n\ge0$ marked points defined over the complex field $\bC$. We denote by  $\cM^1_{7,n;4}$ the locus of points inside $\cM_{7,n}$ representing curves carrying a $g^1_4$. It is classically known that $\cM^1_{7,n;4}$ is irreducible of dimension $17+n$. We prove in this paper that $\cM^1_{7,n;4}$ is rational for $0\le n\le 11$.
\end{abstract}

\section{Introduction and notation}\xlabel{sIntroduction}
Throughout this paper we will work over the field $\bC$ of complex numbers and we will denote by ${\cM}_{g,n}$ the (coarse) moduli space of smooth curves of genus $g\ge2$ with $n\ge0$ marked points defined over $\bC$.

It is well known that ${\cM}_{g,n}$ is irreducible of dimension $3g-3+n$. It thus makes sense to deal with its birational properties, such as its rationality, unirationality, Kodaira dimension $\kappa({\cM}_{g,n})$ and so on. 

For instance, the problem of the rationality of ${\cM}_{g,n}$ has been object of intensive study in the last 30 years (and not only), giving rise to a very long series of papers both in the unpointed and in the pointed case. More precisely, it is known that ${\cM}_{g,n}$ is rational if either $2\le g\le6$ and $n=0$ (see respectively \cite{Ig}, \cite{Ka4}, \cite{SB1} and \cite{Ka2}, \cite{Ka3}, \cite{SB2}), or $1\le g\le 6$ and $1\le n\le \sigma(g)$ (see \cite{Be}, \cite{C--F}, \cite{B--C--F}: in this last paper the numerical function $\sigma(g)$ is defined), while $\kappa({\cM}_{g,n})\ge0$ when $g=1,4,5,6$ and $n\ge\tau(g)$ (see again \cite{Be} and \cite{Lg}: for the definition of $\tau(g)$ see \cite{B--C--F}).

When $g\ge7$ the picture is much more complicated. It is known that ${\cM}_{g,n}$ is unirational for $7\le g\le 14$ and $0\le n\le\sigma(g)$ (see \cite{A--C1}, \cite{A--S}, \cite{Sn}, \cite{C--R}, \cite{Ve}, \cite{Lg} and again \cite{B--C--F}, only to quote the most recent references), while $\kappa({\cM}_{g,n})\ge0$ when either $7\le g\le 14$ and $n\ge\tau(g)$ (see \cite{Lg}, \cite{Fa3}) or $g\ge22$ without restrictions on $n$ (see \cite{H--M}, \cite{Hr}, \cite{E--H}, \cite{Fa1}, \cite{Fa2}, \cite{Fa4}).

Nevertheless ${\cM}_{g,n}$ turns out to contain many interesting (uni)rational loci defined in terms of existence of particular linear series. E.g., one can consider the natural {\sl gonality stratification}\/
$$
\cM_{g,n;2}^1\subseteq\cM_{g,n;3}^1\subseteq\cM_{g,n;4}^1\subseteq\dots\subseteq\cM_{g,n;{\left\lceil\frac{g+2}{2}\right\rceil}}^1=\cM_{g,n}
$$
where $\cM^1_{g,n; d}$ is the locus of $d$--gonal pointed curves, i.e.
$$
\cM^1_{g,n; d} := \{\ (C,p_1,\dots,p_n) \in \cM_{g,n}\ \vert\ \text{$C$ is endowed with a $g^1_d$}\  \}\, .
$$
Such loci are irreducible of dimension $2g+2d-5+n$ when $d<\lceil (g+2)/2\rceil$ (see \cite{A--C1} and the references therein for $n=0$. When $n\ge1$ the irreducibility of $\cM^1_{g,n; d}$ is part of folklore: for an easy proof of this fact see \cite{CsTrigPoint}). In particular, the general point of $\cM^1_{g,n; d}$ does not lie in $\cM^1_{g,n; d-1}$. It follows that the general point of $\cM^1_{g,n; d}$ carries a base--point--free $g^1_d$ and no $g^1_{d-1}$.

The unirationality of these loci is classically known for $d\le 5$ (see e.g. \cite{A--C1} and the references therein: see also \cite{Sch1}). 
Thus it is quite natural to ask whether the strata of such a stratification are rational. 

The rationality of {\sl hyperelliptic stratum}  $\cM_{g,n;2}^1$  was proved in \cite{Ka1}, \cite{Ka5} and \cite{Bo--Ka} when $n=0$ and in \cite{CsHypPoint} when $1\le n\le 2g+8$. In the case $n=0$, the proof essentially rests on the well--known equivalence between hyperelliptic curves of genus $g$ and $(2g+2)$--tuples of unordered points on in $\p1$ up to projective equivalence. 

Some other general rationality results are known for the {\sl trigonal stratum}  $\cM_{g,n;3}^1$. In this case the rationality is known when either $n=0$ and $g\equiv 2\pmod4$ (see \cite{SB1}) or $n=0$ and $g$ is odd (see the recent paper \cite{Ma}) or $1\le n\le 2g+7$ and every $g$ (see \cite{CsTrigPoint}): in all the cases the key point is the classical representation of a trigonal curve as a trisecant divisor on an embedded ruled surface.

Thus the birational description of  $\cM_{7,n;4}^1$ acquires a particular interest for two different reasons. On the one hand,  $\cM_{7,n;4}^1$ has codimension $1$ inside $\cM_{7,n}$, thus its rationality can be viewed  as a sort of \lq\lq suggestion\rq\rq\ of the possible rationality of $\cM_{7,n}$. On the other hand, it is the first necessary step for the possible proof of the rationality of the {\sl tetragonal stratum} $\cM_{g,n;4}^1$ of the aforementioned gonality stratification.

The main result of  the present paper is the  following theorem.

\begin{theorem}
\xlabel{mainTh}
The tetragonal locus $\cM_{7,n;4}^1$ is rational for $0\le n\le 11$.
\end{theorem}

We now quickly describe the content of the paper. In the first section we prove Theorem \ref{mainTh} in the pointed case $n\ge1$. The proof is based on the classical representation of a general curve $C$ of genus $7$ as a plane septic $\overline{C}$ with $8$ nodes in general position. When $C$ is tetragonal, then the eight nodes of $\overline{C}$ are the base locus of a pencil of cubic curves whose residual ninth intersection again lies on $\overline{C}$. This remark is at the base of the construction of an easy birational model of $\cM_{7,n;4}^1$ as quotient of a suitable projective bundle modulo the action of an algebraic group acting on it. The rationality of such a quotient then follows from more or less standard classical results in invariant theory.

In Section \ref{sRationalityPointed} we recall some introductory facts about $g^1_4$ on a curve of genus $7$. In particular, we show that a curve $C$ of genus $7$ represents a general point of the tetragonal stratum if and only if it carries a unique $g^3_8$ which is also very ample. 

In Section \ref{sProjectiveModels}, using such a very ample $g^3_8$, we are able to describe the general tetragonal curve $C$ of genus $7$ as the residual intersection of  a pencil of cubics having a common line. With the help of such an embedded geometric model and a lot of non--trivial representation theory, we are thus able to prove Theorem \ref{mainTh} in the case $n=0$ in Section \ref{sRationalityM704}.

\subsection{Notation}
We work over the complex field $\bC$. In particular, all the algebraic groups ($\GL_k$, $\SL_k$, $\PGL_k$ and so on) are always assumed to have coefficients in $\bC$.

If $g_1,\dots,g_h$ are elements of a certain group $G$ then $\langle
g_1,\dots,g_h\rangle$ denotes the subgroup of $G$ generated by $g_1,\dots,g_h$.

If $V$ is a vector space, then we denote by $\bP(V)$ the corresponding projective space. In particular we set  $\p n:=\bP(\bC^{\oplus n+1})$. 

We denote isomorphisms by $\cong$ and birational equivalences by $\approx$.

For other definitions, results and notation we always refer to \cite{Ha}.

\section{The rationality of ${\cM}_{7,n;4}^1$, $1\le n\le11$}\xlabel{sRationalityPointed}
Let us consider $X_n:=(\p2)^{n-1}\times X$, where $X\subseteq S^8\p2$ is the set of unordered $8$--tuple of points $N:=N_1+\dots+N_8$ in $\p2$ which are base loci of pencils of cubics in the plane. $X_n$ is open and dense in $(\p2)^{n-1}\times S^8\p2$. The incidence variety
\begin{align*}
\bP_n:=\{\ &(\overline{C}, A_1,\dots,A_{n-1},N_1+\dots+N_8)\in\vert\Ofa{\p2}(7)\vert\times X_n\ \vert\ \\
&\text{$N_i$ is double on $\overline{C}$, $A_j\in\overline{C}$, the pencil of cubics through $N_1,\dots,N_8$}\\
 &\text{ has its residual base point $A_{N_1+\dots+N_8}\in \overline{C}$}\ \}
\end{align*}
is naturally endowed with a  structure of projective bundle $p_n\colon\bP_n\to X_n$ with typical fiber $\p{11-n}$. Moreover, there is a natural action of $\PGL_3$ on $\bP_n$. 

Let us consider the general point $(\overline{C}, A_1,\dots,A_{n-1},N_1+\dots+N_8)\in\bP_n$. Let $\nu\colon C\to\overline{C}$ be the normalization map. $C$ is naturally endowed with an ordered set of $n$ points $p_j:=\nu^{-1}(A_j)$, $j\le n-1$, $p_{n}:=\nu^{-1}(A_{N_1+\dots+N_8})$. The curve $C$ cannot be hyperelliptic, since it is endowed with a $g^2_7$ (see \cite{C--M}, Proposition 2.2 (ii)). Thus, the pencil of cubics through $N_1,\dots,N_8$ cut out the fixed point $A_{N_1+\dots+N_8}\in \overline{C}$ plus a complete $g^1_4$ on $C$, say $\vert D\vert$. On the one hand, thanks to \cite{C--K}, Theorem 3.2, the curve $C$ does not carry any other $g^1_4$. On the other hand $\overline{C}$ is the image of $C$ via a morphism associated to
\begin{align*}
\vert \overline{C}\cdot\ell\vert&=\vert 4\overline{C}\cdot\ell-N_1-\dots-N_8-(3\overline{C}\cdot\ell-N_1-\dots-N_8-A_{N_1+\dots+N_8})-A_{N_1+\dots+N_8}\vert=\\
&=\vert K_C-D-p_n\vert\, ,
\end{align*}
i.e. $\overline{C}\subseteq\vert K_C-D-p_n\vert\check{\ }\cong\p2$.

In particular, we have a natural rational map
$$
m_n\colon {\bP}_n\dashrightarrow {\cM}_{7,n;4}^1\, .
$$
\begin{lemma}
\xlabel{lMapmn}
The map $m_n$ is dominant and it induces a birational isomorphism
$$
{\cM}_{7,n;4}^1\approx {\bP}_n/\PGL_{3}\, .
$$ 
\end{lemma}
\begin{proof}
We first describe the fibers of $m_n$. If 
$$
(\overline{C'}, A_1',\dots,A_{n-1}',N_1'+\dots+N_8'),\qquad(\overline{C''}, A_1'',\dots,A_{n-1}'',N_1''+\dots+N_8'')
$$
have the same image $(C,p_1,\dots,p_n)$, then there is a birational map $\varphi\colon \overline{C'}\dashrightarrow\overline{C''}$. The corresponding automorphism on $C$ must fix $K_C$, the unique $g^1_4$ and the points $p_1,\dots,p_n$. Thus $\varphi$ is induced by a projectivity of $\vert K_C-D-p_n\vert\check{\ }\cong\p2$.  It follows that the fibers of $m_n$ are exactly the orbits of the action of $\PGL_3$ on ${\bP}_n$. 

In particular ${\bP}_n/\PGL_{3}\approx\im(m_n)\subseteq{\cM}_{7,n;4}^1$. In order to complete the proof it suffices to check that $\dim(\im(m_n))=\dim({\bP}_n/\PGL_{3})=\dim({\cM}_{7,n;4}^1)$. We have
$$
17+n=\dim({\cM}_{7,n;4}^1)\ge\dim(\im(m_n))\ge \dim({\bP}_n)-\dim(\PGL_{3})=25+n-8=17+n\, ,
$$
thus $\dim(\im(m_n))=17+n=\dim({\cM}_{7,n;4}^1)$. It follows that  $\im(m_n)$ is dense inside ${\cM}_{7,n;4}^1$.
\end{proof}

We now go to prove the Theorem \ref{mainTh} for $1\le n\le 11$ making use of the aforementioned representation. We first examine the case $5\le n\le 11$ which is the easiest one. To this purpose let $E_1:=[1,0,0]$, $E_2:=[0,1,0]$, $E_3:=[0,0,1]$, $E_4:=[1,1,1]$ and consider the subset
$$
Y_n:=\{\ (E_1,E_2,E_3,E_4,A_5,\dots,A_{n-1},N_1+\dots+N_8)\in X_n\ \}\, .
$$
It is trivial to check that $Y_n$ is a $(\PGL_3, \: \mathrm{id})$--section of $X_n$ in the sense of \cite{Ka1}. The scheme $p_n^{-1}(Y_n)$ is a projective bundle on $Y_n$, thus it is trivially irreducible and rational. It follows from Proposition 1.2 of \cite{Ka1} that
$$
{\cM}_{7,n;4}^1\approx{\bP}_n/\PGL_3\approx p_n^{-1}(Y_n)\, ,
$$
thus ${\cM}_{7,n;4}^1$ is also rational for $5\le n\le11$.

Now we turn our attention to the slightly more difficult case $1\le n\le 4$. Let 
\begin{equation*}
{\mathcal E}:=\{\ (E, A)\in\vert\Ofa{\p2}(3)\vert\times \p2\ \vert\ A\in E\ \}\, .
\end{equation*}
The projection $\mathcal E\to \p2$ endows $\mathcal E$ with a natural structure of Zariski locally trivial projective bundle over $\p2$ with fiber isomorphic to $\p8$. Thus $\mathcal E$ is rational and  $\dim(\mathcal E)=10$. There exists a natural dominant rational map $q_n\colon {\bP}_n\to\mathcal E$ defined by
$$
(\overline{C}, A_1,\dots,A_{n-1},N_1+\dots+N_8)\mapsto(E_{N_1+\dots+N_8},A_{N_1+\dots+N_8})\, ,
$$
where $E_{N_1+\dots+N_8}$ is the unique cubic tangent to $\overline{C}$ at $A_{N_1+\dots+N_8}$ through the points $N_1,\dots,N_8$. The fiber of $q_n$ over $(E,A)$ is birationally isomorphic to $\p{10-n}\times(\p2)^{n-1}\times \vert\Ofa E(3E\cdot L-A)\vert$ (here $L$ is a general line in $\p2$). In particular, the map $q_n$ can be factorized into a sequence of Zariski locally trivial projective bundles.

Trivially $q_n$ is $\PGL_3$ equivariant and we have the following more or less classical result.

\begin{lemma}
\xlabel{almost free}
The action of $\PGL_3$ on ${\mathcal E}$ is almost free.
\end{lemma}
\begin{proof}
It is a well--known classical result (see e.g. \cite{Brieskorn}, Theorem 4 of Section II.7.3) that, up to projectivities, the equation of each general cubic curve can be put in its Hesse form, i.e.
$$
t_1^3+t_2^3+t_3^3-3\lambda t_1t_2t_3=0\, ,
$$
where $\lambda\in\bC$ satisfies $\lambda^3\ne1$. The subgroup of projectivities fixing such a polynomial is the extended Heisenberg group
$$
H(3)^e:=\left\langle
\left(
\begin{array}{ccc}
0 & 1 & 0\\
0 &  0 & 1\\
1 & 0 &  0
\end{array}
\right),\left(
\begin{array}{ccc}
1 & 0 & 0\\
0 & \zeta & 0\\
0 & 0 & \zeta^2
\end{array}
\right),\left(
\begin{array}{ccc}
1 & 0 & 0\\
0 &  0 & 1\\
0 & 1 &  0
\end{array}
\right)\right\rangle\cong(\bZ_3\times\bZ_3)\rtimes\bZ_2
$$
(see Section II.7.3 of  \cite{Brieskorn}). Hence the set ${\mathcal A}:=\{\ A\in\p2\ \vert\ \exists \varphi\in H(3)^e,\ \varphi\ne\mathrm{id}:\ \varphi(A)=A\ \}$ is the union of a finite number of points and lines in $\p2$. Thus ${\mathcal P}:=\p2\setminus\mathcal A$ is open and dense. Moreover, $\mathcal P$ has non--empty intersection with each irreducible curve $E$ in the Hesse pencil. 

By construction, for each point $A\in \mathcal P\cap E$, the pair $(E,A)$ has trivial stabilizer inside $\PGL_3$. 
\end{proof}

We can draw the commutative diagram

\xycenter{
	\PP_n \ar[d]_{q_n} \ar@{-->}[r]
	& \PP_n /\PGL_3 \approx {\cM}_{7,n;4}^1 \ar[d]^{\bar{q}_n}
	\\
	\mathcal{E} \ar@{-->}[r]
	 & \mathcal{E}/\PGL_3. 
	}
Let
$$
\widehat{\mathcal{ E}}:=\{\ (f, x) \in \mathrm{Sym}^3 ( \CC^3)^{\vee } \times \CC^3 \ \vert\ f(x) =0\ \}\, .
$$
Thus $\mathcal{E}\cong\widehat{\mathcal{ E}}/T$ where $T:=(\bC^*)^2$ acts almost freely on $\widehat{\mathcal{ E}}$ via homotheties. The group $\SL_3$ acts naturally on $\widehat{\mathcal{ E}}$. Such an action induces on $\mathcal{E}$ the natural action of $\PGL_3$. It follows that the map $\widehat{\mathcal{ E}}/\SL_3\dashrightarrow \mathcal{E}/\SL_3\cong \mathcal{E}/\PGL_3$ has a section. Moreover $\SL_3$ is special and acts almost freely on $\widehat{\mathcal{ E}}$, thus also the natural projection $\widehat{\mathcal{ E}}\dashrightarrow\widehat{\mathcal{ E}}/\SL_3$ has a section. Composing these two sections with the natural projection $\widehat{\mathcal{ E}}\to\mathcal E$ we finally obtain a section $\sigma\colon \mathcal E/\PGL_3\dashrightarrow \mathcal{E}$ of the bottom map of the commutative square above.

Recall that the map $q_n$ is a sequence of Zariski locally trivial projective bundles, so from the existence of $\sigma$ it follows that $\PP_n/\mathrm{PGL}_3$ is a tower of Zariski locally trivial projective bundles over $\mathcal{E}/\mathrm{PGL}_3$ too. Since $\mathcal{E}/\mathrm{PGL}_3$  is rational by Castelnuovo's theorem, we conclude that the same is true also for $\bP_n/\PGL_3\approx{\cM}_{7,n;4}^1$.

\

Thus we have completed the proof of Theorem \ref{mainTh} for $1\le n\le 11$. In order to complete its proof, it remains to analyze the difficult case $n=0$. We devote the remaining part of this paper to the description of this case

\section{Projective models of tetragonal curves of genus $7$}\xlabel{sProjectiveModels}
Let $C$ be a curve of genus $g$. Assume the existence of two curves $C_i$ of genera $g_i$ and morphisms $\varphi_i\colon C\to C_i$ of respective degrees $d_i$, $i=1,2$. If $\varphi_1$ and $\varphi_2$ are not composed with the same pencil, then the following Castelnuovo--Severi formula holds true (see [A--C--G--H], Exercise VIII C 1):
\begin{equation}
\label{castelnuovo severi}
g\le (d_1-1)(d_2-1)+d_1g_1+d_2g_2\, .
\end{equation}

Assume that $C$ is a curve of genus $7$ carrying a base--point--free $g^1_4$, say $\vert D\vert$. Then $h^0\big(C,\Ofa C(2D)\big)\ge3$. We say that $\vert D\vert$ is of type $I$ if $h^0\big(C,\Ofa C(2D)\big)=3$, of type $II$ otherwise. Notice that $C$ cannot carry any $g^1_3$. Otherwise, Formula \ref{castelnuovo severi} would yield $7=g\le 6$, since  $\vert D\vert$ and the $g^1_3$ cannot be composed with the same pencil.
We have the following possible cases.
\begin{enumerate}
\item $C$ is hyperelliptic. Thanks to Formula \ref{castelnuovo severi} the linear system $\vert D\vert$ is composed with the involution: in particular, $C$ is endowed with infinitely many $g^1_4$'s.
\item $C$ is bielliptic. Again Formula \ref{castelnuovo severi} yields that the linear system$\vert D\vert$ is composed with the involution: in particular, $C$ is endowed with infinitely many $g^1_4$'s.
\item $C$ is neither hyperelliptic nor bielliptic, but it carries more than one $g^1_4$
\item $C$ carries exactly one $g^1_4$ of type $II$.
\item $C$ carries exactly one $g^1_4$ of type $I$.
\end{enumerate}

We now go to characterize the different types of curves in terms of existence of particular $g^3_d$ on them.

\begin{proposition}
\xlabel{plane model}
The curve $C$ carries either a finite number $t\ge2$ of $g^1_4$'s or exactly one $g^1_4$ of type $II$ if and only if $C$ is birationally isomorphic to a plane sextic with three, possibly infinitely near, double points.
\end{proposition}
\begin{proof}
If $C$ is birationally isomorphic to a plane sextic $\overline{C}$ with at most double points as singularities, then $C$ is neither hyperelliptic nor bielliptic due to \cite{C--M}, Proposition 2.2 (ii) and (v).  If $\overline{C}$ has  at least two double points, then it is clear that $C$ is endowed with at least two $g^1_4$. Assume that $\overline{C}$ has a unique double point, say $N_1$. The arithmetic genus of a plane sextic is $10$, hence Clebsch's formula for the genus of a plane curve yields
$$
\sum_{p\in\overline{C}}{\mu_p(\overline{C})(\mu_p(\overline{C})-1)}=6,
$$
$\mu_p(\overline{C})$ being the multiplicity of $p\in \overline{C}$. I follows the existence of a second double point $N_2$ infinitely near to $N_1$. The linear system of conics through $N_1$ and $N_2$ has dimension $3$ cut out $\vert\Ofa C(2D)\vert$ on $\overline{C}$ outside $N_1$ and $N_2$. Thus $h^0\big(C,\Ofa C(2D)\big)\ge4$, i.e. $C$ carries a $g^1_4$ of type $II$ in this case.

Conversely, each pair of distinct $g^1_4$ on $C$ induces a morphism $\varphi\colon C\to\p1\times\p1\subseteq \p3$. Its image $\overline{C}\subseteq\p1\times\p1$ is a curve whose bidegree is necessarily either $(2,2)$ or $(4,4)$. In the first case $\varphi$ has degree $2$, whence $C$ would be hyperelliptic, a contradiction, since $C$ carries a finite number of $g^1_4$. In the second case $\varphi$ would be birational, whence $\overline{C}$ would carry at least a double point. Projecting on a plane from such a double point we obtain a plane sextic birationally isomorphic to $C$. Clebsch's formula for the genus of a plane curve of degree $6$ implies that $\overline{C}$ has either a triple point or it carries three, possibly  infinitely near, double points. If $\overline{C}$ carries a triple point, then it would be endowed with a $g^1_3$, a contradiction. We conclude that the singularities of $\overline{C}$ are at most double points.

Finally assume that $C$ is endowed with a unique $g^1_4$ of type $II$, say $\vert D\vert$. In this case, $\vert 2D\vert$ is a $g^r_8$ with $r\ge3$. Each such $g^r_8$ is necessarily special, thus Clifford's Theorem yields $r\le4$. If equality holds then $C$ would be hyperelliptic, thus it would carry infinitely many $g^1_4$, a contradiction. It follows that $r=3$. In particular, if $f_1,f_2\in H^0\big(C,\Ofa C(D)\big)$ is a basis, then there exists an element $g\in H^0\big(C,\Ofa C(2D)\big)$, such that $f_1^2,f_1f_2,f_2^2,g$ form a basis of $H^0\big(C,\Ofa C(2D)\big)$. Thus we have a map $\varphi\colon C\to\p3$ given by $(x_0,x_1,x_2,x_3)=(f_1^2,f_1f_2,f_2^2,g)$. It follows that the image $\overline{C}$ is contained in the quadric cone $x_0x_2-x_1^2=0$. In particular, the arithmetic genus of $\overline{C}$ is $9$ (see \cite{Ha}, Exercise V.2.9), thus $\overline{C}$ has at least a double point and we can repeat the argument used for curves carrying at least two distinct $g^1_4$.
\end{proof}

Now we characterize curves carrying exactly one $g^1_4$ of type $I$. We first state and prove a technical lemma which will be useful later on in the paper.

\begin{lemma}
\xlabel{cubic pencil}
If $C\subseteq\p3$ is a curve of degree $8$ and genus $7$, then there exist two cubic surfaces $F_1$, $F_2$ and a quartic surface $G$ through $C$, such that $C=F_1\cap F_2\cap G$. Moreover the general cubic surface through $C$ is smooth.
\end{lemma}
\begin{proof}
If $C$ lies on a smooth quadric, then its bidegree $(a,b)$ would satisfy the conditions $a+b=8$ and $(a-1)(b-1)=7$. If $C$ lies on a quadric cone, then we should have $8=2a$ and $7=(a-1)^2$ (see \cite{Ha}, Exercise V.2.9). In both the cases the two equations have no common integral solutions. Thus the minimal degree of a surface through $C$ is at least $3$.

The cohomology of the standard exact sequence $0\to\mathcal{I}_C\to\Ofa{\p3}\to\Ofa C\to0$ twisted by $3$ shows that 
$$
h^0\big(\p3,\mathcal{I}_C(3)\big)\ge2\, .
$$
Thus we can find two cubic surfaces $F_1$ and $F_2$ such that $F_1\cap F_2=C\cup L$. Due to degree reasons $L$ is a line, thus $C$ is aCM (see \cite{Mi}, Theorem 5.3.1). It follows that $h^0\big(\p3,\mathcal{I}_C(3)\big)=2$, $h^0\big(\p3,\mathcal{I}_C(4)\big)=9$ and the homogeneous ideal of $C$ is minimally generated by two cubic $F_1$, $F_2$ and a single quartic surfaces $G$ (see \cite{Mi}, Proposition 5.2.10). In particular, $C=F_1\cap F_2\cap G$. 

The general cubic surface $F$ through $C$ is smooth outside $C\cup L$ due to Bertini's theorem. Since $C$ and $L$ are smooth and they are exactly the intersection of $F_1$ and $F_2$ outside $C\cap L$, it also follows that such a general cubic  $F$ is smooth outside $C\cap L$. At the points of $C\cap L$, $C$ is smooth, thus it is the intersection of two surfaces at those points. The two cubics $F_1$ and $F_2$ are tangent at the points of $C\cap L$, thus one can always take $G$ and one of the cubics. In particular, at least one among the cubics  $F_1$ and $F_2$ is smooth at the points of $C\cap L$. We conclude that we can assume that the general cubic $F$ through $C$ is smooth everywhere.
\end{proof}

We are now ready to characterize curves carrying a unique $g^1_4$ of type $I$.

\begin{proposition}
\xlabel{space model}
The curve $C$ carries exactly one $g^1_4$ of type $I$ if and only if it is endowed with a very ample $g^3_8$.
\end{proposition}
\begin{proof}
Consider a divisor $D$ on $C$ and let $K_C$ a canonical divisor on $C$. Then $\vert D\vert$ is a $g^1_4$ (resp. a $g^3_8$) if and only if $\vert K_C-D\vert$ is a $g^3_8$ (resp. a $g^1_4$). Thus $C$ carries a unique $g^1_4$ if and only if it carries a unique $g^3_8$.

Now assume that $C$ is endowed with a unique base--point--free $g^1_4$ of type $I$, say $\vert D\vert$. We will show that $\vert K_C-D\vert$ is very ample. To this purpose we have first to show that it is base--point--free. For each $p\in C$, we have 
$$
h^0\big(C,\Ofa C(K_C-D-p)\big)=h^0\big(C,\Ofa C(K_C-D)\big)
$$
if and only if $\vert D+p\vert$ is a $g^2_5$. Thus $C$ would have a plane model of degree at most $5$, hence its genus would be at most $6$, a contradiction. Now we have to show that $\vert K_C-D\vert$ separates points and tangent vectors. For possibly coinciding points $p_1,p_2\in C$, we have 
$$
h^0\big(C,\Ofa C(K_C-D-p_1-p_2)\big)=h^0\big(C,\Ofa C(K_C-D)\big)-1
$$
if and only if $\vert D+p_1+p_2\vert$ is a $g^2_6$. Consider the associated morphism $\varphi\colon C\to \overline{C}\subseteq \p2$. We have three possible cases for the pair $(\deg(\varphi),\deg(\overline{C}))$, namely $(1,6)$, $(2,3)$, $(3,2)$. In the lat case $C$ would be trigonal. In the second case $C$ would be either bielliptic or hyperelliptic. The first case cannot occur due to Proposition \ref{plane model} above, since $C$ is endowed with a unique $g^1_4$ of type $I$.

Conversely let us assume that $C$ is endowed with a very ample $g^3_8$. Thus $C$ carries a $g^1_4$, say $\vert D\vert$. We know that the very ample $g^3_8$ is $\vert K_C-D\vert$. We have to show that such a $g^1_4$ is of type $I$ and it is unique. 

Assume that $\vert D\vert$ is of type $II$. Let us identify $C$ with its image via the map associated to the $g^3_8$ in $\p3$. We know from Proposition \ref{space model} that there is a line $L\subseteq \p3$ such that $C\cup L$ is the complete intersection of a smooth cubic surface $F$ with another cubic surface. Let us consider the standard representation of $F$ as the blow up of $\p2$ at six general points. Let $\ell$ be the strict transform on $F$ of a general line of $\p2$ and let $e_1,\dots,e_6$ be the exceptional divisors in the blow up. We can always assume that $L=e_1$ so that
$$
C\in \vert 9\ell-4e_1-3e_2-\dots-3e_6\vert\, .
$$
By adjunction on $S$ the canonical system $\vert K_C\vert$ on $C$ is cut out by $\vert 6\ell-3e_1-2e_2-\dots-2e_6\vert$, thus $\vert D\vert$ is cut out on $C$ by the pencil $\vert 3\ell-2e_1-e_2-\dots-e_6\vert$, which coincides with the pencil cut out on $S$ by the planes through $L$. 

Notice that $\vert 2D\vert$ is cut out on $C$ by  $\vert 6\ell-4e_1-2e_2-\dots-2e_6\vert$. The cohomology of the standard exact sequence $0\to\Ofa S(-C)\to\Ofa S\to\Ofa C\to0$ twisted by $\Ofa S(6\ell-4e_1-2e_2-\dots-2e_6)$, gives the exact sequence
\begin{align*}
0\longrightarrow H^0\big(S,\Ofa S(-3\ell+e_2+\dots+e_6)\big)&\longrightarrow H^0\big(S,\Ofa S(6\ell-4e_1-2e_2-\dots-2e_6)\big)\longrightarrow \\
\longrightarrow H^0\big(C,\Ofa C(2D)\big)&\longrightarrow H^1\big(S,\Ofa S(-3\ell+e_2+\dots+e_6)\big)\, .
\end{align*}
Since $\vert -3\ell+e_2-\dots+e_6\vert$ cannot be effective, the first space is zero. Thanks to Serre's duality $h^2\big(S,\Ofa S(-3\ell+e_2+\dots+e_6)\big)=h^0\big(S,\Ofa S(e_1)\big)=1$. Riemann--Roch theorem finally yields $h^1\big(S,\Ofa S(-3\ell+e_2+\dots+e_6)\big)=0$. 

Trivially $\vert 3\ell-2e_1-e_2-\dots-e_6\vert$ contains a smooth integral curve $A$. Since $(3\ell-2e_1-e_2-\dots-e_6)^2=0$, it follows that $\Ofa A(A)\cong\Ofa A$. Thus the cohomology of $0\to\Ofa S\to\Ofa S(A)\to\Ofa A(A)\to0$, yields $h^0\big(S,\Ofa S(3\ell-2e_1-e_2-\dots-e_6)\big)=2$. Since
$$
(6\ell-4e_1-2e_2-\dots-2e_6)\cdot(3\ell-2e_1-e_2-\dots-e_6)=0
$$
we deduce that each effective divisor in $\vert6\ell-4e_1-2e_2-\dots-2e_6\vert$ is the sum of two elements in $\vert3\ell-2e_1-e_2+\dots-e_6\vert$. It follows that  
$$
h^0\big(C,\Ofa C(2D)\big)=h^0\big(S,\Ofa S(6\ell-4e_1-2e_2-\dots-2e_6)\big)=3
$$
i.e.  $\vert D\vert$ is of type $I$.

Assume that $\vert D\vert$ is not unique and let  $\vert D'\vert$ be a distinct $g^1_4$ on $C$. Since $C$ carries a $g^3_8$, due to \cite{C--M}, Proposition 2.2 (ii) and (v), it follows that $C$ is not hyperelliptic nor bielliptic. As in the proof of Proposition \ref{plane model} on checks that $\vert D+D'\vert$ is a $g^3_8$. Then $\vert K_C-D-D'\vert$ is a $g^1_4$, say $\vert D''\vert$. In particular, the linear system of planes of $\p3$ cut out on $C$ exactly the linear system $\vert D'+D''\vert$. It follows that the planes through the divisor $D'\in\vert D'\vert$ cut out $\vert D''\vert$ on $C$. It follows that the support of $D'$ must be contained in a line $L'$. Since $\deg(D')=4$ such a line is contained in each cubic through $C$, thus $L'=L$, whence $\vert D''\vert=\vert D\vert$. Thus $h^0\big(C,\Ofa C(2D)\big)=2+h^0\big(C,\Ofa C(D')\big)=4$, i.e. $\vert D\vert$ should be of type $II$. But this contradicts what we proved above.
\end{proof}

\section{The rationality of ${\cM}_{7,0;4}^1$}\xlabel{sRationalityM704}
We fix a line in $L\subseteq\p3$ and we consider the set of pencils of cubic surfaces through $L$ which is a Grassmannian $G(2,16)$. If $\Sigma\in G(2,16)$ , then its base locus ocontains $L$. If $\Sigma$ is general, then it contains a smooth cubic surface $F$ and the residual intersection is a smooth curve $C$ on $F$, whence its  genus is $7$. Thus $C$ is endowed with a base--point--free $g^3_8$. In particular we have a natural rational map
$$
m_0\colon G(2,16)\dashrightarrow {\cM}_{7,0;4}^1
$$
whose image is the locus of tetragonal curves of genus $7$ endowed with a unique $g^1_4$ of type $I$, thanks to Proposition 3.3. Let $\PGL_{4,L}$ be the stabilizer of $L$ inside $\PGL_4$. It is clear that there is a natural action of $\PGL_{4,L}$ on $G(2,16)$. The orbits of such an action are trivially contained in the fibers of $m_0$.

\begin{lemma}
\xlabel{m0}
The map $m_0$ is dominant and it induces a birational isomorphism
$$
{\cM}_{7,0;4}^1\approx G(2,16)/\PGL_{4,L}\, .
$$ 
\end{lemma}
\begin{proof}
We first describe the fibers of $m_0$.
Let $\Sigma'$ and $\Sigma''$ be two general pencils of cubic surfaces with base loci $C'\cup L$ and $C''\cup L$ respectively. Since  $\Sigma'$ and $\Sigma''$ are general, $C'$ and $C''$ are smooth curve of genus $7$. If $m_0(\Sigma')=m_0(\Sigma'')$, then the base loci of $\Sigma'$ and $\Sigma''$ must be abstractly isomorphic. Such an isomorphism must induces an isomorphism $\varphi\colon C'\to C''$. If $\vert D'\vert$ and $\vert D''\vert$ are the unique $g^1_4$ on $C'$ and $C''$ respectively, then we have $\varphi^*\vert D''\vert=\vert D'\vert$. Since also  $\varphi^*\vert K_{C''}\vert=\vert K_{C'}\vert$ holds, we finally deduce that $\varphi^*\vert K_{C''}-D''\vert=\vert K_{C'}-D'\vert$.  In particular $\varphi$ sends the very ample $g^3_8$ on $C''$ into the very ample $g^3_8$ on $C'$, thus it induces a projectivity of the whole $\p3$ transforming $\Sigma'$ in $\Sigma''$. Such a projectivity must fix $L$. It follows that the fibers of $m_0$ are exactly the orbits of the action of $\PGL_{4,L}$ on $G(2,16)$,

In particular $G(2,16)/\PGL_{4,L}\approx\im(m_0)\subseteq{\cM}_{7,0;4}^1$. In order to complete the proof it suffices to check that $\dim(\im(m_0))=\dim(G(2,16)/\PGL_{4,L})=\dim({\cM}_{7,0;4}^1)$. We have
$$
17=\dim({\cM}_{7,0;4}^1)\ge\dim(\im(m_0))\ge \dim(G(2,16))-\dim(\PGL_{4,L})=28-11=17\, ,
$$
thus $\dim(\im(m_0))=17=\dim({\cM}_{7,0;4}^1)$. We conclude that $\im(m_0)$ is dense inside ${\cM}_{7,0;4}^1$.
\end{proof}

Let $\SL_{4,L}$ be the stabilizer of $L$ in $\SL_4$. We are interested in the rationality properties of the quotient $G(2, 16)/\SL_{4,L}\cong G(2,16)/\PGL_{4,L}$. 

\begin{theorem}\xlabel{tCubics}
The space $X=G(2, 16)/ \SL_{4,L}$ is rational.
\end{theorem}

The proof requires several preparatory results. The next lemma shows that $X$ is birational to a linear group quotient and collects all the facts about this representation which are needed in the sequel.

\

\begin{lemma}\xlabel{lInitialReduction}
The space $X = G (2, \: 16) / \mathrm{SL}_{4, L}$ is birational to the quotient $R/ G$ where $R$ is a linear representation of the linear algebraic group $G$, and $R$, $G$ and the action are defined below: 
\begin{itemize}
\item[(1)]
\textbf{(Group structure)}. One has
\[
G = G_R \ltimes U
\]
where $G_R = \mathrm{GL}_2  \times G_R'$, the group $G_R'\subseteq\mathrm{SL}_4$ being the subgroup consisting of matrices
\[
\left( \begin{array}{cc} A_2 & 0 \\ 0 & A_3 \end{array} \right) \, , \quad A_2, \: A_3 \in \mathrm{Mat}_{2\times 2} (\CC ),
\]
hence as an abstract group a central product $(\CC^{\ast})\cdot (\SL_2 \times \SL_2)$. Here the torus $\CC^{\ast}$ is embedded in $\SL_4$ via
\[
\lambda \mapsto \mathrm{diag} (\lambda, \lambda, \lambda^{-1}, \lambda^{-1}) \, . 
\]
The group $U$ is given by $U = (\CC^2)^{\veeÊ} \otimes \CC^2 = \mathrm{Hom} (\CC^2 , \: \CC^2 )$, viewed as an abelian algebraic group under addition of homomorphisms. As a normal subgroup in the semidirect product, elements $u\in U$ are acted on by $(A_1, \: A_2, \: A_3 ) \in G_R$, where $A_1 \in \mathrm{GL}_2 $, and $A_2$, $A_3$ are as before, in the following way:
\[
(A_1,\:  A_2, \: A_3 ) \cdot u = A_2 u A_3^{-1} \, .
\]

\item[(2)]
\textbf{(Structure of the representation)}. 
 The representation $R$ is a representation with a Jordan-H\"older filtration
\[
R_0 :=0 \subset R_1 \subset R_2 \subset R_3 = R
\]
with completely reducible quotients $Q_i=R_{i}/R_{i-1}$ which as representations of $G_R$ are given by
\begin{align*}
Q_1 &= \CC^2 \otimes \mathrm{Sym}^3 \CC^2 \otimes \CC, \\
Q_2 &= \CC^2 \otimes \CC^2\otimes \mathrm{Sym}^2 \CC^2 ,\\
Q_3 &= \CC^2 \otimes  \mathrm{Sym}^2 \CC^2\otimes \CC^2\, , 
\end{align*}
with the action of $G_R$ being given by the tensor product action of $(A_1, \: A_2, \: A_3)$ in the three factors. The action of $U$ is induced by the standard $G_R$-equivariant maps
\begin{align*}
(\CC^2)^{\vee } \otimes \CC^2&  \to  \mathrm{Hom} (\CC^2 \otimes \mathrm{Sym}^2 \CC^2\otimes \CC^2, \: \CC^2 \otimes \CC^2 \otimes \mathrm{Sym}^2 \CC^2 ), \\
(\CC^2)^{\vee } \otimes \CC^2&  \to  \mathrm{Hom} (\CC^2 \otimes \CC^2 \otimes \mathrm{Sym}^2 \CC^2 , \: \CC^2 \otimes \mathrm{Sym}^3 \CC^2 \otimes \CC )
\end{align*}
given by contraction and multiplication in the second and third factors. 
\item[(3)]
\textbf{(Ineffectivity kernel)} The ineffectivity kernel $I$ of the action of $G$ on $R$ is as an abstract group $I\simeq \ZZ/4\ZZ$, and generated by
\[
(A_1, \: A_2, \: A_3) = \left (  \left( \begin{array}{cc} i & 0 \\ 0 & i\end{array} \right) , \: \left( \begin{array}{cc} i & 0 \\ 0 & i\end{array} \right) , \: \left( \begin{array}{cc} i & 0 \\ 0 & i\end{array} \right)   \right) \, .
\]
The group acting effectively will be denoted by $\bar{G} = G/(\ZZ/4\ZZ)$. 
\end{itemize}
\end{lemma}

\begin{proof}
Note that for $X = G (2, \: 16) / \mathrm{SL}_{4, L}$, the group $\mathrm{SL}_{4, L}$ is nothing but $G_R' \ltimes U$, and $X$ is birational to (we view the stabilized line $L\subset \CC^4$ as a two-plane)
\[
\left(  \mathrm{Hom} ( \CC^2 , \; \mathrm{Sym}^3 (\CC^4)^{\vee } /\mathrm{Sym}^3 (L)^{\veeÊ} \right) / \mathrm{GL}_2 \times \SL_{4,L}\, ,
\]
which in turn is birational to 
\[
(\CC^2 \otimes \mathrm{Sym}^3 (\CC^4)^{\vee } /\mathrm{Sym}^3 (L)^{\veeÊ} )/ \mathrm{GL}_2 \times \SL_{4,L}
\]
(because of the self-duality of $\CC^2$ as $\mathrm{SL}_2 $-representation, we get birational quotients if we choose as $\mathrm{GL}_2 $-representation in the first factor of the tensor product $\CC^2$ or $(\CC^2)^{\vee }$; we prefer $\CC^2$ for simplicity, but this is not essential for the following). The representation $\CC^2 \otimes \mathrm{Sym}^3 (\CC^4)^{\vee } /\mathrm{Sym}^3 (L)^{\veeÊ} $ has a filtration as a $G$-module with quotients $Q_1$, $Q_2$, $Q_2$, $Q_3$ as claimed in (2) above. This proves (1) and (2).\\
For the determination of the ineffectivity kernel in (3), we refer to the stronger statement Lemma \ref{lGenericallyFreeTwoStep} below proven with the help of computer algebra, and which implies in particular part (3) of the present lemma. 
\end{proof}

\begin{remark}\xlabel{rRepresentationTheoryS4}
We have to recall the representation theory of $\mathfrak{S}_4$ viewed as the group of permutations of four
letters $\{ a,\: b,\: c,\: d\}$ for subsequent use. The character table is (cf. \cite{Sr}).\\
\begin{center}
\begin{tabular}{c | rrrrr}
 & $1$ & $(ab)$ & $(ab)(cd)$ & $(abc)$ & $(abcd)$  \\ \hline
$\chi_0$       & $1\;$ & $1\;$ & $1\;$ & $1\;$ & $1\;$ \\
$\epsilon$     & $1\;$ & $-1\;$ & $1\;$ & $1\;$ & $-1\;$ \\
$\theta$       & $2\;$ & $0\;$ & $2\;$ & $-1\;$ & $0\;$ \\
$\psi$         & $3\;$ & $1\;$ & $-1\;$ & $0\;$ & $-1\;$ \\
$\epsilon\psi$ & $3\;$ & $-1\;$ & $-1\;$ & $0\;$ & $1\;$ 
\end{tabular}
\end{center}\vspace{0.3cm}
$V_{{\chi}_0}$ is the  trivial $1$--dimensional representation, $V_{\epsilon}$ is the $1$--dimensional
representation where $\epsilon (g)$ is the sign of the permutation $g$; $\mathfrak{S}_4$ being the semidirect
product of $\mathfrak{S}_3$ by the Klein four group, $V_{\theta}$ is the irreducible two-dimensional
representation induced from the representation of $\mathfrak{S}_3$ acting on the elements of $\mathbb{C}^3$ which
satisfy $x+y+z=0$ by permutation of coordinates. $V_{\psi}$ is the representation on the elements of $\CC^4$ with $x+y+z+w=0$ by permutation of coordinates; finally, $V_{\epsilon\psi}=V_{\epsilon}\otimes V_{\psi}$.
\end{remark}

The $G_R$-representation $Q_3$ in item (2) of Lemma \ref{lInitialReduction} is one with a nontrivial stabilizer in general position $H$. The determination of $H$ and its normalizer $N(H)$ in $G_R$ is carried out in the next lemma.  It gives us a $(G, \: N(H) \ltimes U)$-section in $R$. 

\begin{lemma}\xlabel{lStabilizer}
Consider the action of $G_R$ on $Q_3 = R_3/R_2 = \CC^2 \otimes \mathrm{Sym}^2 \CC^2 \otimes \CC^2$. 
\begin{itemize}
\item[(1)] \textbf{(Structure of the stabilizer in general position)}. 
The stabilizer in general position $H$ of the representation $Q_3= R_3/R_2$ inside the group $G_R$, consisting of matrices $(A_1, \: A_2, \: A_3)$ as in Lemma \ref{lInitialReduction},  can be described as follows: put
\[
A = \left( \begin{array}{cc} i & 0 \\ 0 & -i  \end{array}\right)  , \qquad B = \left(   \begin{array}{cc} 0 & 1 \\ -1 & 0 \end{array}\right)\, .
\]
Then $A$ and $B$ generate inside $\mathrm{SL}_2 $ a nontrivial central extension of the Klein four group $\ZZ/2\ZZ^2 \subset \mathrm{PSL}_2 $, which is a finite Heisenberg group (or extraspecial $2$-group in different terminology) which we denote by $\tilde{H}$ 
\[
0 \to \ZZ /2\ZZ \simeq  \{ \pm 1 \}  \to \tilde{H} \to (\ZZ /2\ZZ )^2 = \langle \overline{A} , \overline{B} \rangle \to 0
\]
(here $\overline{A}$ and $\overline{B}$ denote the classes of $A$ and $B$ in $\mathrm{PSL}_2 $ respectively).

\

Then the stabilizer in general position $H \subset G_R$ (well defined up to conjugacy) has a representative  given by the subgroup of matrices
\[
(A_1, \: A_2, \: A_3 ) = (\lambda h , \: \pm \lambda^{-1} h , \: \lambda h) , \quad \lambda \in \CC^{\ast } , \; h \in \tilde{H} \, .
\]
\item[(2)] \textbf{(Structure of the normalizer)}. 
The normalizer $N(H)$ of $H$ in $G_R$ can be described as follows: define matrices $(\tau, \tau, \tau )$ and $(\sigma, \sigma , \sigma )$ by 
\[
\tau = \left( \begin{array}{cc}  \theta^{-1} & 0 \\ 0 & \theta \end{array}\right), \quad \sigma := \frac{1}{\sqrt{2}} \left(  \begin{array}{cc} \theta^3 & \theta^7 \\ \theta^5 & \theta^5  \end{array}\right)\, .
\]
and $\theta = \mathrm{exp} (2\pi i/8)$. Their classes in $\mathrm{PSL}_2 $ generate the symmetric group $\mathfrak{S}_4$ normalizing the Klein four group, but inside $\mathrm{SL}_2 $, they generate a nontrivial central extension
\[
1 \to \ZZ /2\ZZ \simeq \{ \pm 1 \} \to \tilde{\mathfrak{S}}_4 \to \mathfrak{S}_4 \to 1 \, .
\]
Then $N(H)$ contains a subgroup abstractly isomorphic to $\tilde{\mathfrak{S}}_4 \ltimes \tilde{H}$ where: $\tilde{\mathfrak{S}}_4$ is embedded in $N(H)$ diagonally as the matrices
\[
(A_1, \: A_2, \: A_3 ) = (x, \: x, \: x ) , \: x\in \tilde{\mathfrak{S}}_4 \, ;
\]
and the additional copy of $\tilde{H}$ in the semidirect product $\tilde{\mathfrak{S}}_4 \ltimes \tilde{H}$ is embedded in the second factor as matrices
\[
(A_1, \: A_2, \: A_3 ) = (\mathrm{id} , \: y, \: \mathrm{id}) , \: y\in \tilde{H} \, .
\]
The complete normalizer $N(H)$ is generated by $\tilde{\mathfrak{S}}_4 \ltimes \tilde{H}$ and the center of $G_R$ consisting of matrices
\[
(A_1, \: A_2, \: A_3 ) = \left(  \left( \begin{array}{cc}  \lambda & 0 \\ 0 & \lambda \end{array}\right) , \;  \left( \begin{array}{cc}  \pm \mu & 0 \\ 0 & \pm \mu \end{array}\right) , \: \left( \begin{array}{cc}  \mu^{-1} & 0 \\ 0 & \mu^{-1} \end{array}\right) \right) \, .
\]

\

\item[(3)] \textbf{(Structure of the $(G, \; N(H) \ltimes U)$-section)}.
By (1) and (2) there is a $(G, \; N(H) \ltimes U)$-section $R'$ in $R$ which is a linear representation of $N(H) \ltimes U$ with a Jordan-H\"older filtration $0 = R_0' \subset R_1' \subset R_2' \subset R_3' = R'$ with completely reducible quotients $Q_i' = R_i'/R_{i-1}'$ given by 
\[
Q_1' =Q_1 = \CC^2 \otimes \mathrm{Sym}^3 \CC^2 , \; Q_2' = Q_2= \CC^2 \otimes \CC^2 \otimes \mathrm{Sym}^2 \CC^2 , \; Q_3'= Q_3^H\, .
\]

The space of $H$--invariants in $R_3/R_2$ is an $N(H)$--section for the action of $\prod_{i=1}^3 \SL_2$. This space $(R_3/R_2)^H$ has dimension $3$. The action of the copy of $\tilde{\mathfrak{S}}_4$ in $N(H)$ on $Q_3^H$ is via the standard representation of $\mathfrak{S}_3$. That is, in the notation introduced in Remark \ref{rRepresentationTheoryS4},
\[
Q_3^H = V_{\chi_0} \oplus V_{\theta } \, .
\]
\end{itemize}
\end{lemma}

All assertions of Lemma \ref{lStabilizer} have been double-checked independently via computer algebra by the second author to avoid mistakes. The Macaulay 2 scripts can be found at {\ttfamily http://xwww.uni-math.gwdg.de/bothmer/tetragonal/ }. Below we give a proof not relying on this. 

\begin{proof}[Proof of Lemma 4.5]
\emph{Step 1. Determination of the stabilizer in general position}.
Recall the accidental isomorphisms of Lie groups $\mathrm{Spin}_3 \simeq \SL_2$ and $\mathrm{Spin}_4  \simeq \SL_2 \times \SL_2$. Under these isomorphisms we have the correspondences of representations
\[
\CC^3 \simeq \mathrm{Sym}^2 \CC^2\, , \qquad \CC^4 \simeq \CC^2 \otimes \CC^2\, .
\]
Thus $R_3/R_2$ is isomorphic to $\CC^3\otimes\CC^4$ and the group $\mathrm{Spin}_3 \times \mathrm{Spin}_4 $ acts as $\mathrm{SO}_3 \times \mathrm{SO}_4 $ in this representation. The table in \cite{APopov78} shows firstly that the stabilizer in general position inside $\mathrm{SO}_3\times \mathrm{SO}_4$ of this representation is isomorphic to $(\ZZ/2\ZZ )^2$; and secondly, that it maps isomorphically to the stabilizer in general position for the action of $(\mathrm{SO}_3 \times \mathrm{SO}_4 )/$(center) on $\PP (\CC^3 \otimes \CC^4 )$.  Hence, to prove (1) of Lemma \ref{lStabilizer} it suffices to show that the subgroup of matrices
\[
(A_1, \: A_2, \: A_3) = (h, \: \pm h, \: h), \; h \in \tilde{H}
\]
of order $16$ coincides with the preimage of the stabilizer in general position inside $\mathrm{SO}_3  \times \mathrm{SO}_4 $ in $\mathrm{SL}_2 \times \mathrm{SL}_2 \times \mathrm{SL}_2$. Note that a stabilizer in general position is only well-defined up to conjugacy, and we are in particular interested in finding a good model for it and the associated invariant subspace.

\ 

Every general $3\times 4$ complex matrix can -by the action of $\mathrm{SO}_3\times \mathrm{SO}_4$- be brought to the normal form
\begin{gather*}
\left(
\begin{array}{cccc}
\lambda_1 & 0 & 0 & 0\\
0 & \lambda_2 & 0 & 0\\
0 & 0 & \lambda_3 & 0
\end{array}
\right)\, ,
\end{gather*}
the $\lambda_i$ being pairwise different. This follows from the polar decomposition for complex matrices and, afterwards, the fact that symmetric complex matrices which are similar are orthogonally similar. See \cite{Ga}, Chapter XI, \S 2, Thm. 3 and Thm. 4.

Thus the stabilizer $(\ZZ/2\ZZ )^2$ in $\mathrm{SO}_3\times \mathrm{SO}_4$ consists of pairs of matrices
\[
(\mathrm{diag}(\epsilon_1, \epsilon_2, \epsilon_3 ) , \; \mathrm{diag}(\epsilon_1, \epsilon_2, \epsilon_3 , 1) )
\]
where the $\epsilon_i$ are either $+1$ or $-1$ subject to the condition that the determinant of $\mathrm{diag}(\epsilon_1, \epsilon_2, \epsilon_3 )$ is $1$. The invariant subspace of this $(\ZZ /2\ZZ )^2$ has dimension $3$.

\

Note that the quadratic form that $\SL_2\times \SL_2$ fixes on $\CC^2\otimes \CC^2$ is the determinant of a two by two matrix. An orthogonal system of three vectors for the associated bilinear form can thus be given by the matrices
\[
m_1 = \left( \begin{array}{cc} 1 & 0\\ 0 & 1  \end{array} \right), \qquad m_2 = \left( \begin{array}{cc} 1 & 0\\ 0 & -1  \end{array} \right), \qquad m_3 = \left( \begin{array}{cc} 0 & 1\\  1 & 0  \end{array} \right)\, .
\]
The bilinear form that $\SL_2$ fixes on $\mathrm{Sym}^2 (\CC^2)$ is given by contraction of conics (interpret one of them as a dual conic via the isomorphism $\mathrm{Sym}^2 \CC^2 \simeq \mathrm{Sym}^2 (\CC^2)^{\vee }$). Hence an orthogonal basis can be given by 
\[
q_1 = x^2+y^2, \;  q_2 = x^2 - y^2 , \; q_3 =xy \, . 
\] 
We thus have to find the stabilizer of 
\begin{gather*}
\lambda_1 m_1\otimes q_1 + \lambda_2 m_2\otimes q_2 + \lambda_3 m_3 \otimes q_3 \\
 = \left( \begin{array}{cc}  
 \lambda_1 (x^2 + y^2) + \lambda_2 (x^2-y^2)  & \lambda_3 (xy)  \\
 \lambda_3 (xy)   &  \lambda_1 (x^2+ y^2 )- \lambda_2 (x^2-y^2)
 \end{array}\right)
\end{gather*}
inside $\SL_2 \times \SL_2\times\SL_2$ for general $\lambda_1, \: \lambda_2, \: \lambda_3$. Let
\[
A = \left( \begin{array}{cc} i & 0 \\ 0 & -i  \end{array}\right)  , \qquad B = \left(   \begin{array}{cc} 0 & 1 \\ -1 & 0 \end{array}\right)\, .
\]
Then 
\begin{gather*}
A\left( \begin{array}{cc}  a & b \\ c& d\end{array} \right) A^t =  \left( \begin{array}{cc}  -a & b \\ c& -d\end{array} \right), \\\
B\left( \begin{array}{cc}  a & b \\ c& d\end{array} \right) B^t =  \left( \begin{array}{cc}  d & -c \\ -b& a\end{array} \right)
\end{gather*}
hence we see that the elements
\[
(A, A, A ), \; (B, B, B) \in \SL_2 \times \SL_2 \times\SL_2 
\]
stabilize the matrix $\lambda_1 m_1\otimes q_1 + \lambda_2 m_2\otimes q_2 + \lambda_3 m_3 \otimes q_3 $ given above! These generate a Klein four group in $(\SL_2  \times \SL_2 \times\SL_2 ) / (-\mathrm{id}, -\mathrm{id} , -\mathrm{id})$, but as they only commute up to the central element $z:= (-\mathrm{id}, -\mathrm{id} , -\mathrm{id})$ in $\SL_2 \times \SL_2 \times\SL_2 $ they generate there a nontrivial central extension 
\[
0 \to \ZZ /2\ZZ = \langle z \rangle  \to \tilde{H} \to (\ZZ /2\ZZ )^2 = \langle \overline{(A,A,A)} , \overline{(B,B,B)} \rangle \to 0 
\]
which is a particular (finite) Heisenberg group in the sense of Mumford's Tata Lectures on Theta III \cite{Mum-Tata3}. There is one further $\ZZ /2\ZZ$ in the stabilizer in general position: the center of the copy of $\SL_2$ acting on $\mathrm{Sym}^2 (\CC )$; this establishes (1) of the Lemma.

\

\emph{Step 2. Determination of the normalizer.} Note first that $G_R$ does contain a copy of $\tilde{\mathfrak{S}}_4$ as indicated in (2), normalizing $H$, and that the full automorphism group of the Klein four group coincides with $\mathfrak{S}_4$. Hence it suffices to determine the matrices $(A_1, \: A_2, \: A_3) \in G_R$ which, via conjugation, act as the identity on the Heisenberg group of matrices
\[
(h, \: h, \: h ) , \; h \in \tilde{H}
\]
inside $G_R$ \emph{modulo the subgroup of $G_R$ generated by elements} $(\lambda , \: \pm \lambda^{-1} , \: \lambda )$. But the conditions for a matrix $M$ to commute with $A$ above up to some nonzero constant factor $\lambda$ imply that $M$ equals
\[
\left( \begin{array}{cc} m_1 & 0 \\ 0 & m_4 \end{array} \right) , \; \lambda = +1 , \quad \mathrm{or} \quad \left( \begin{array}{cc} 0 & m_2 \\ m_3 & 0 \end{array} \right) , \; \lambda = -1
\]
and if one of the matrices of the preceding two shapes also commutes with $B$ up to a (possibly different) nonzero factor $\lambda'$, we must have
\[
M = \left( \begin{array}{cc} m_1 & 0 \\ 0 & m_1 \end{array} \right) \quad \mathrm{or} \quad \left( \begin{array}{cc} m_1 & 0 \\ 0 & -m_1 \end{array} \right)
\]
or 
\[
M = \left( \begin{array}{cc} 0 & m_2 \\ m_2 & 0 \end{array} \right) \quad \mathrm{or} \quad \left( \begin{array}{cc} 0 & m_2 \\ -m_2 & 0 \end{array} \right)\, ,
\]
that is- modulo the center of $G_R$- we must have that $(A_1, \: A_2, \: A_3)$ is in $\tilde{H} \times \tilde{H} \times \tilde{H} \subset\mathrm{SL}_2  \times \mathrm{SL}_2 \times \mathrm{SL}_2$.  But then the assertion of (2) follows at once, as elements in $\tilde{H}$ commute up to a sign, but the sign change must be the same in the first and third factors, and an element in $\tilde{H}$ anticommutes with every other element in $\tilde{H}$ except itself.

\

\emph{Step 3. Structure of $(R_3/R_2)^H$.} The only statement of Lemma \ref{lStabilizer}, (3) which still needs some explanation is that as $\mathfrak{S}_4$-representation $(R_3/R_2)^H = V_{\chi_0} \oplus V_{\theta }$. This is most readily seen by going back to the $\mathrm{SO}_3 \times \mathrm{SO}_4$--picture introduced at the beginning of the proof, where it is obvious, namely $\mathfrak{S}_4$ acts as $\mathfrak{S}_3$ via ordinary permutations on $\lambda_1$, $\lambda_2$, $\lambda_3$.
\end{proof}

We compute the decomposition of the Jordan-H\"older quotients as $\mathfrak{S}_4$-representations. 

\begin{lemma}\xlabel{lS4Decomposition}
As $\mathfrak{S}_4$-representations we have
\begin{gather*}
\CC^2 \otimes \CC^2 \otimes \mathrm{Sym}^2 \CC^2 = V_{\chi_0} \oplus V_{\theta } \oplus V_{\psi } \oplus 2 V_{\epsilon\psi } , \\
\CC^2 \otimes \mathrm{Sym}^3 \CC^2 = V_{\theta } \oplus V_{\psi } \oplus V_{\epsilon \psi}\, .
\end{gather*}
Here we view these as $\mathfrak{S}_4$-representations via the embedding $\tilde{\mathfrak{S}}_4\subset \SL_2 \times\SL_2 \times \SL_2$ described in Lemma \ref{lStabilizer}. 
\end{lemma}

\begin{proof}
We start by decomposing $\CC^2\otimes \CC^2$ viewed as two by two matrices 
\begin{gather*}
M=\left( \begin{array}{cc}  a & b \\ c& d \end{array} \right) \, 
\end{gather*}
under the action of $\mathfrak{S}_4$. Using the formulas for $A$ and $B$ in Lemma \ref{lStabilizer}, we find that $M$ can only be invariant under the Klein four group if it is antisymmetric. But an antisymmetric matrix $M$ is also invariant under $\sigma$ and $\tau$. Thus the space of $\mathfrak{S}_4$--invariants in $\CC^2 \otimes \CC^2$ is one-dimensional. The three-dimensional subspace of matrices spanned by $m_1$, $m_2$, $m_3$ (notation as in the proof of Lemma \ref{lStabilizer}) is $\mathfrak{S}_4$--invariant. The trace of an element in the Klein four-group, $A$ say, on this space is easily calculated to be $-1$, and the trace of the four-cycle $\tau$ on this space (remark $\tau^4= \sigma^3=1$) is similarly found to be $+1$. Thus
\[
\CC^2 \otimes \CC^2 = V_{\chi_0} \oplus V_{\epsilon \psi }\, .
\]
Now $\mathrm{Sym}^2 \CC^2$ has no $A$--invariants. Hence it is either $V_{\psi }$ or $V_{\epsilon\psi }$. The trace of $\tau$ on it is $+1$. So $\mathrm{Sym}^2 \CC^2 = V_{\epsilon\psi }$. Checking characters on both sides shows the decomposition
\begin{gather*}
V_{\epsilon\psi } \otimes V_{\epsilon\psi}Ê= V_{\chi_0} \oplus V_{\theta } \oplus (V_{\thetaÊ} \otimes V_{\epsilon\psi})
= V_{\chi_0} \oplus V_{\theta } \oplus (V_{\psi } \oplus V_{\epsilon\psi } )\, .
\end{gather*}
Note by the way that this also implies
\begin{gather*}
(\CC^2 \otimes \CC^2 ) \otimes (R_3/R_1)^H = \CC^2 \otimes \mathrm{Sym}^2 \CC^2 \otimes \CC^2
\end{gather*}
as $\mathfrak{S}_4$--representations as it should be. This checks the computation. Hence we have seen that
\[
\CC^2 \otimes \mathrm{Sym}^2 \CC^2 \otimes  \CC^2 = V_{\chi_0} \oplus V_{\theta } \oplus V_{\psi } \oplus 2 V_{\epsilon\psi }\, .
\]
The $\mathfrak{S}_4$--representation $\CC^2 \otimes \mathrm{Sym}^3\CC^2$ can be decomposed as follows. One has
\begin{gather*}
\CC^2 \otimes \CC^2 \otimes \mathrm{Sym}^2 \CC^2 = \CC^2 \otimes ( \CC^2 + \mathrm{Sym}^3 \CC^2 ) = \CC^2 \otimes \CC^2 + \CC^2 \otimes \mathrm{Sym}^3 \CC^2 \, 
\end{gather*}
as $\mathfrak{S}_4$--representations. Hence
\[
 \CC^2 \otimes \mathrm{Sym}^3 \CC^2  = V_{\theta } \oplus V_{\psi } \oplus V_{\epsilon \psi} \, .
\]
\end{proof}

The proof of the main Theorem \ref{tCubics} now requires some further lemmas. In what follows we say
that a variety $X$ is stably rational of level $t$ if $X\times \PP^t$ is
rational.

\begin{lemma}\xlabel{lGenericallyFreeTwoStep}
The action of $(N(H) \ltimes U)/I$, where $I$ is the ineffectivity kernel $I \simeq \ZZ/4\ZZ$ of the action described in lemma \ref{lInitialReduction}, item (3), is already generically free on the two step quotient $R_3'/R_1'$ of $R'$. Hence, by the no-name lemma, $R'/ (N(H) \ltimes U)$ is generically a vector bundle of rank $8$ over $(R_3'/R_1')/ (N(H) \ltimes U )$. Thus the proof of Theorem \ref{tCubics} reduces to the proof that 
\[
(R_3'/R_1')/ (N(H) \ltimes U)
\]
is stably rational of level at most $8$. 
\end{lemma}

\begin{proof}
The only point to check is the generic freeness of the action of $N(H) \ltimes U$ on $R_3'/R_1'$. This can been done by a computer algebra calculation, the commented Macaulay 2 scripts are at {\ttfamily http://xwww.uni-math.gwdg.de/bothmer/tetragonal/ }.
\end{proof}

The following Lemma contains a standard trick to reduce from $N(H) \ltimes U$ to $N(H)$. 

\begin{lemma}\xlabel{lFurtherReduction}
The quotient $(R_3'/R_1')/ (N(H) \ltimes U)$ is stably rational of level $8$ if $(R_3'/R_1')/ N(H)$ is stably rational of level $4$.
\end{lemma}

\begin{proof}
Consider the representation $E$ of $N(H) \ltimes U$ which is a two-step extension (of dimension $5$)
\[
0 \to (\CC^2)^{\vee } \otimes \CC^2 \to E \to \CC\, .
\]
Then $(R_3'/R_1')(N(H) \ltimes U) \times \CC^5 \simeq ((R_3'/R_1') \oplus E) / (N(H) \ltimes U)\simeq (R_3'/R_1') / N(H) \times \CC$.
\end{proof}

\begin{lemma}\xlabel{lBasicRationalityGroup}
The group $G_R$ has a natural representation $V$ which is 
\[
\CC^2 \otimes \CC^2 + \CC^2 \otimes \CC^2 + \CC^2 \otimes \mathrm{Sym}^2 \CC^2 \otimes \CC^2
\]
and $(A_1, \: A_2, \: A_3)\in G_R$ acts via 
\[
(A_1, \: A_2, \: A_3) \cdot (M_1, \: M_2,  \: x) = (A_1M_1A_2^{-1} , \: A_1M_2 A_3^{-1} , \: (A_1,\:  A_2,\:  A_3) \cdot x)
\]
where $M_1$ and $M_2$ are interpreted as matrices in $\CC^2 \otimes \CC^2$ and $x$ is some element in $\CC^2 \otimes \mathrm{Sym}^2 \CC^2 \otimes \CC^2$, and the action there is the same action considered already above. Then the quotient $V/G_R$ is rational. Moreover the representation of $N(H)/I$ given by
\[
V'= \CC^2 \otimes \CC^2 + \CC^2 \otimes \CC^2 + Q_3^H
\]
(obtained by taking a section in $V$) is generically free for $N(H)/I$ (and the quotient is birationally equivalent to the former, hence also rational). Hence $N(H)/I$ has a generically free representation of dimension $11$ with rational quotient. 
\end{lemma}

\begin{proof}
There is a $(G_R, \: G')$-section in the representation $V$ which is $\{ \mathrm{id} \} \times \{ \mathrm{id} \} \times \CC^2 \otimes \mathrm{Sym}^2 \CC^2 \otimes \CC^2 \simeq \CC^2 \otimes \mathrm{Sym}^2 \CC^2 \otimes \CC^2$ where $\mathrm{id} \in \CC^2 \otimes \CC^2$ is the identity $2\times 2$ matrix, and $G'$ is the subgroup of $G_R$ consisting of matrices $(A_1, \: A_2, \: A_3)$ with $A_1=A_2=A_3$. That is, $G'$ is the subgroup of $\mathrm{GL}_2 $ generated by $\mathrm{SL}_2 $ and $\mathrm{diag} (i, \: i)$. It follows firstly that $V/G_R$ is rational as $\PP (\CC^2 \otimes \mathrm{Sym}^2 \CC^2 \otimes \CC^2 ) / \mathrm{SL}_2 $ is rational (recall that each linear representation of $\mathrm{SL}_2\times\CC^\ast$ has rational quotient, thanks to some well--known theorems by Bogomolov and Kastylo: e.g., see \cite{Bo--Ka},\cite{Ka1}, \cite{Ka5} and the other references cited therein); and secondly, that $V$ is generically free for $G_R/I$, hence that $V'$ is generically free for $N(H)/I$, because $\CC^2 \otimes \mathrm{Sym}^2 \CC^2 \otimes \CC^2$ is generically free for $G'/\langle \mathrm{diag} (i, \. i) \rangle$. This concludes the proof of the lemma. 
\end{proof}

Hence we can can conclude as follows.

\begin{proof}[Proof of Theorem \ref{tCubics}]
It suffices, by Lemma \ref{lFurtherReduction}, to establish stable rationality of level $4$ of $(R_3'/R_1')/(N(H)/I)$. By Lemma \ref{lBasicRationalityGroup}, it suffices to find a generically free $N(H)/I$-subrepresentation of $R_3'/R_1'$ with a complement in $R_3'/R_1'$ of dimension $\ge 7$, using the no-name Lemma. Now
\[
R_3'/R_1' = Q_3^H + \CC^2 \otimes \CC^2 \otimes \mathrm{Sym}^2 \CC^2\, .
\] 
The representation $\CC^2 \otimes \CC^2 \otimes \mathrm{Sym}^2 \CC^2$ has an $N(H)$-invariant summand which is $\CC^2 \otimes \CC^2$, an $N(H)$-invariant complement being $\CC^2 \otimes \mathrm{Sym}^3 \CC^2$. In fact, the representation $\CC^2 \otimes \CC^2 \otimes \mathrm{Sym}^2 \CC^2$ is, as an $\tilde{\mathfrak{S}}_4 \ltimes \tilde{H}$-representation a tensor product of two $\tilde{\mathfrak{S}}_4 \ltimes \tilde{H}$-representations (recall also from Lemma \ref{lStabilizer}, item (2), that $N(H)$ is generated by $\tilde{\mathfrak{S}}_4 \ltimes \tilde{H}$ and the center of $G_R$): the one representation is a representation of $\tilde{\mathfrak{S}}_4 \ltimes \tilde{H}/\tilde{H} \simeq \tilde{\mathfrak{S}}_4$, namely $\CC^2 \otimes \mathrm{Sym}^2 \CC^2 \simeq \CC^2 + \mathrm{Sym}^3 \CC^2$ corresponding to grouping the first and third factor in $\CC^2 \otimes \CC^2 \otimes \mathrm{Sym}^2 \CC^2$; the other representation is $\CC^2$ (corresponding to the factor in $\CC^2 \otimes \CC^2 \otimes \mathrm{Sym}^2 \CC^2$ in the middle), which is an $\tilde{\mathfrak{S}}_4 \ltimes \tilde{H}$-representation via the inclusions
\[
\tilde{\mathfrak{S}}_4 \ltimes \tilde{H} \subset N(H) \subset \mathrm{SL}_2  \times \mathrm{SL}_2 \times \mathrm{SL}_2
\]
followed by the projection $\mathrm{pr}_2\, :\, \mathrm{SL}_2  \times \mathrm{SL}_2 \times \mathrm{SL}_2 \to \mathrm{SL}_2$ to the second factor. Hence, as every subspace of $\CC^2 \otimes \CC^2 \otimes \mathrm{Sym}^2 \CC^2$ is stabilized by the center of $G_R$, we indeed have an $N(H)$-invariant splitting
\[
\CC^2 \otimes \CC^2 \otimes \mathrm{Sym}^2 \CC^2 \simeq \CC^2 \otimes \CC^2 + \CC^2 \otimes \mathrm{Sym}^3 \CC^2
\]
Then $(R_3/R_2)^H + \CC^2 \otimes \CC^2$ is generically free for $N(H)/I$, and the complement $\CC^2 \otimes \mathrm{Sym}^3 \CC^2$ has dimension $\ge 7$. To check the generic freeness, which can be done by hand, note that in Lemma \ref{lS4Decomposition} we saw that $\CC^2 \otimes \CC^2 \simeq V_{\chi_0} \oplus V_{\epsilon \psi }$, and the Klein four group $\tilde{H} / \{ \pm 1 \}$, which is contained in the stabilizer $H/I$ of a general point in $Q_3^H$, consequently acts effectively in $\PP (\CC^2 \otimes \CC^2 )$. 
\end{proof}


\bigskip
\noindent
Christian B\"ohning\\
Fachbereich Mathematik der Universit\"at Hamburg,\\
Bundesstr. 55, 20146 Hamburg, Germany, \\
and \\
Mathematisches Institut der Georg-August-Universit¬at G¬ottingen,\\
Bunsenstr. 3-5, 37073 G\"ottingen, Germany\\
e-mail: {\tt christian.boehning@math.uni-hamburg.de}\\
e-mail: {\tt boehning@uni-math.gwdg.de}

\bigskip
\noindent
Hans-Christian Graf von Bothmer\\
Mathematisches Institut der Georg-August-Universit¬at G¬ottingen,\\
Bunsenstr. 3-5, 37073 G\"ottingen, Germany\\
e-mail: {\tt bothmer@uni-math.gwdg.de}

\bigskip
\noindent
Gianfranco Casnati\\
Dipartimento di Matematica,  Politecnico di Torino, \\
corso Duca degli Abruzzi 24, 10129 Torino, Italy \\
e-mail: {\tt gianfranco.casnati@polito.it}

\end{document}